\documentclass[reqno,11pt]{amsart}

\usepackage{amsmath,amsfonts,amssymb,amsthm,epsfig}
\usepackage{tikz-cd}
\usetikzlibrary{arrows}

 \allowdisplaybreaks
\numberwithin{equation}{section}

\font\script=rsfs10 at 11pt
\def\eps{\varepsilon}
\def\H{{\mbox{\script H}\,\,}}

\def\F{\mathfrak F}
\def\En{\mathfrak R}
\def\R{\mathbb R}
\def\S{\mathbb S}

\def\angle#1#2#3{#1\widehat{#2}#3}
\def\bal{\begin{aligned}}
\def\eal{\end{aligned}}
\def\proofof#1{\begin{proof}[Proof of #1]}

\def\XXint#1#2#3{{\setbox0=\hbox{$#1{#2#3}{\int}$} \vcenter{\vspace{-1pt}\hbox{$#2#3$}}\kern-.5\wd0}}

\def\comp{\subset\subset}

\newcounter{mt}
\def\maintheorem#1#2#3{\par \medskip \noindent {\bf Theorem~\mref{#1}}~(#2).~{\it #3}\par}
\def\mref#1{\Alph{#1}}
\def\maintheoremdeclaration#1{\stepcounter{mt}\newcounter{#1}\setcounter{#1}{\arabic{mt}}}

\maintheoremdeclaration{main}

\newtheorem{theorem}{Theorem}[section]
\newtheorem{lemma}[theorem]{Lemma}

\begin{document}

\title[Minimality of balls in the small volume regime\dots]{Minimality of balls in the small volume regime for a general Gamow type functional}

\author{D. Carazzato}
\author{N. Fusco}
\author{A. Pratelli}

\begin{abstract}
We consider functionals given by the sum of the perimeter and the double integral of some kernel $g:\R^N\times\R^N\to \R^+$, multiplied by a ``mass parameter'' $\eps$. We show that, whenever $g$ is admissible, radial and decreasing, the unique minimizer of this functional among sets of given volume is the ball as soon as $\eps\ll 1$.
\end{abstract}

\maketitle

\section{Introduction}

The celebrated ``liquid drop model'' for the atomic nucleus, introduced in the '30s by Gamow, consists in the minimization of the functional
\[
P(E) + \iint_{E\times E} \frac 1{|y-x|^{N-\alpha}}\,dy\,dx
\]
among sets of given volume in $\R^N,\, N\geq 2$, where $0<\alpha<N$ is a given parameter. Even though the physically relevant case is $N=3,\, \alpha=2$, when the  second term is the Coulombic energy, this more general functional has been deeply investigated since then, both by physicists and  mathematicians. There is a clear competition between the two terms in the energy, since the ball at the same time minimizes the perimeter, by the isoperimetric inequality, and maximizes the second term, by the Riesz inequality. More in general, the first term favours concentration of mass, while the second one favours disgregation. An important peculiarity of the model is that the two terms scale differently, in particular the perimeter is the leading term for sets of small volume, while the Riesz energy (i.e., the second term) is the leading term for large volumes. By rescaling, instead of considering different masses, it is equivalent but mathematically more convenient to consider only sets of given volume, say $\omega_N$, the measure of the unit ball, and consider the modified functional
\[
P(E) + \eps\, \iint_{E\times E} \frac 1{|y-x|^{N-\alpha}}\,dy\,dx
\]
for some positive $\eps$. As said above, it is then clear that minimizers become closer to the ball when $\eps\searrow 0$, while they tend to disperse completely when $\eps\nearrow +\infty$.\par

In fact, physicists always took for granted that minimizers are \emph{exactly} balls if $\eps$ is sufficiently small. This property has been investigated by several mathematicians and  proved in a series of recent papers. More precisely, Kn\"upfer and Muratov~\cite{KM1,KM2} proved that balls are the only minimizers for $\eps\ll 1$ when $N=2$, and when $3\leq N\leq 7$ if $1<\alpha<N$, see also the proof given by Julin in \cite{J} in the case $\alpha=2$ . One should  point out that there are big differences between the case $0<\alpha\leq 1$, the so-called \emph{near-field dominated regime}, and the case $1<\alpha<N$, the so-called \emph{far-field dominated regime}. Roughly speaking,  things are much more complicated in the near-field dominated regime since in the Riesz energy the contribution of pairs of points which are very close to each other is fairly strong (the mathematical consequence is that several objects which are controlled in the far-field dominated regime become ill-defined due to some integrals which do not converge). Later on, Bonacini and Cristoferi~\cite{BC} proved the same result for every $N$, still with $1<\alpha< N$. And finally, Figalli, Fusco, Maggi, Millot and Morini~\cite{F2M3} proved the result in any dimension $N\geq 2$ and for every $0<\alpha<N$, even replacing the perimeter $P(E)$ by the fractional perimeter $P_s(E),\, 0<s\leq1$.\par

Countless papers in the last few years  investigated the properties of the Riesz energy if one replaces the kernel $|y-x|^{\alpha-N}$ with $g(y-x)$ for a more general $g:\R^N\setminus \{0\}\to \R^+$, hence considering the ``Riesz-type energy''
\begin{align*}
\En(E)=\En(E,E)\,, && \hbox{where} && \En(F,G) =  \iint_{F\times G} g(z-w)\,dz\,dw\,.
\end{align*}
As soon as $g$ is radial and decreasing, the Riesz inequality still implies that,  among sets of given volume, $\En(E)$ is maximized by the ball. The corresponding ``Gamow-type functional'' is then
\[
\F_\eps(E) = P(E) + \eps\, \En(E)\,.
\]
A natural question is whether it is possible to show that $\F_\eps$ is minimized by balls when $\eps\ll 1$ for more general functions $g$ than the negative powers. Observe that this question is reasonable only if $g$ is radial, hence we will always make this assumption and  write, with a small abuse of notation, $g(t)=g(x)$ for any $x\in\R^N$ with $|x|=t$. We say that a radial function $g:\R^N\setminus \{0\} \to \R^+$ is \emph{admissible} if and only if
\begin{equation}\label{necessary}
\int_0^1 g(t) t^{N-1}\, dt<+\infty\,.
\end{equation}
The meaning of this property is clear. Indeed, without this assumption $\En(E)=+\infty$ for every non-empty set $E$, hence the whole problem makes no sense, but with this assumption all sets of finite volume have finite energy. In the recent paper~\cite{NP} it was proved that, among sets of given volume, balls are the unique minimizers of $\F_\eps$ for $\eps$ small enough in the $2$-dimensional case if $g$ is radial, decreasing and positive definite, which means 
\[
\iint_{\R^N\times\R^N} g(y-x)f(x)f(y)\,dy\,dx \ge 0 \qquad \forall f\in C_c(\R^N).
\]
In particular, in case $g(v)=|v|^{\alpha-N}$, the positive definiteness is equivalent to the assumption $1<\alpha<N$.\par

In this paper, we are able to give a simple proof of the minimality of balls for every radial and decreasing function $g$ and in any dimension $N\geq 2$. Despite considering a much more general function than in the above-mentioned papers, our proof is considerably shorter than those available in the literature.

\maintheorem{main}{Balls are unique minimizers for small $\eps$}{Let $g:\R^N\setminus\{0\}\to \R^+$ be a radially decreasing, admissible function. Then, there exists $\bar\eps>0$ such that, for every $0<\eps<\bar\eps$, the unique minimizer (up to translations) of $\F_\eps$ among sets of volume $\omega_N$ is the unit ball.}

The fact that minimizers of $\F_\eps$ actually exist, for $\eps$ small enough, has been proved in wider generality in~\cite[Proposition~1.2 and Lemma~4.1]{NP}

\section{Proof of the main result}

In the following we  denote by $B(z,r)$ the ball  centered at $z$ with radius $r$. When $r=1$ we simply write $B(z)$ (or $B$ if the center is the origin). Moreover, in order to ease the notation, if   no confusion arises, to  indicate integration with respect to $\H^{N-1}$ measure we shall often write  $dx$ instead of $d\H^{N-1}(x)$.

This section is devoted to the proof of our result. The first step is to observe that the minimizers of the functional $\F_\eps$ converge in ${\rm C}^{1,\gamma}$ to the ball, up to translations. This is a fairly standard fact, see for instance~\cite[Proposition~2.2 and Lemma~3.6]{CL}. More precisely, we have the following result.
\begin{lemma}[Regularity of minimizers]\label{lemreg}
There exists $\eps_1>0$, only depending on $N$ and  $g$, such that for every $0<\eps<\eps_1$ and every minimizer $E$ of $\F_\eps$ under the volume constraint $|E|=\omega_N$, there exists a function $u\in W^{1,2}(\S^{N-1})$ such that, up to a translation,
\begin{equation}\label{defEu}
E= E(u) = \Big\{ z\in \R^N:\, z=\rho x,\, x\in\S^{N-1},\, 0\leq \rho < 1+u(x)\Big\}\,.
\end{equation}
The function $u$ actually belongs to ${\rm C}^{1,\gamma}$ for every $0<\gamma<1/2$, and its norm can be taken arbitrarily small, up to decrease the value of $\eps_1$. Moreover,
\begin{equation}\label{marcogippo}
P(E)\geq P(B) + C_N \|u\|_{W^{1,2}}^2
\end{equation}
for a geometric constant $C_N$, only depending on $N$.
\end{lemma}
\begin{proof}
Recall that, given two constants $\Lambda,\, r_0>0$, a set of finite perimeter  $E\subseteq \R^N$ is said to be a $(\Lambda, r_0)$-\emph{perimeter minimizer} if for every ball $B(z,r)$ with $r<r_0$ and every set $F$ such that $F\Delta E\comp B(z,r)$ one has
\[
P(E)\leq P(F) + \Lambda |E\Delta F|\,.
\]
It is easily checked that if $E_\eps$ is a minimizer of  $\F_\eps$ under the volume constraint $|E_\eps|=\omega_N$, then $E_\eps$ is a $(\Lambda,r_0)$-perimeter minimizer for every small $\eps$, with $\Lambda$ and $r_0$ not depending on $\eps$. A proof of this fact in a more general setting can be found in~\cite[Proposition~3.6]{NP}. Moreover, as already noticed, , the sets $E_\eps$ converge up to translations in $L^1$ to $B$ as $\eps$ goes to $0$. Thus~\cite[Theorem~21.14]{M} implies that they also converge in the Kuratowski sense, hence their boundaries are contained in an arbitrarily small neighborhood of $\partial B$, provided $\eps$ is small enough. Since one can cover $\partial B$ with finitely many cylinders with arbitrarily small excess (see~\cite[Chapter~22]{M} for the definition of the excess), by~\cite[Proposition~22.6]{M} also  $E_\eps$ has arbitrarily small excess in the same cylinders if $\eps$ is small enough. As a consequence, \cite[Theorem~26.3]{M} ensures that $\partial E_\eps$ is ${\rm C}^{1,\gamma}$  for every $0<\gamma<1/2$, with uniform bounds. Again, since $E_\eps$ converges in the $L^1$ sense to $B$ as $\eps$ goes to $0$,  by interpolation the convergence holds also in ${\rm C}^{1,\gamma}$. Thus, for $\eps$ small enough $\partial E_\eps$ is a graph over $\S^{N-1}$, hence there exists a function  $u_\eps$ such that~(\ref{defEu}) holds.\par

Finally, the estimate~(\ref{marcogippo}) is a result by Fuglede, proved in~\cite{F} (see also~\cite{CL}).
\end{proof}

The second observation is that the $W^{1,2}$ norm of a function $u$ controls the double integral of $(u(y)-u(x))^2$ with weight $g$. In other words, we have the following result.

\begin{lemma}\label{genest}
For every Sobolev function $u\in W^{1,2}\big(\S^{N-1}\big)$, we have
\[
\iint_{\S^{N-1}\times\S^{N-1}} g(y-x)\big(u(y)-u(x)\big)^2\, dy\,dx \leq C\|\nabla_\tau u\|_{L^2(\S^{N-1})}^2\,,
\]
where $C$ is a constant, only depending on $N$ and  $g$, and $\nabla_\tau$ stands for the tangential gradient.
\end{lemma}
\begin{proof}
For $x\in\S^{N-1}$, set $E_x=\{\omega\in\S^{N-1}:\omega\perp x\}$. Observe that  any $y\in\S^{N-1}\setminus\{x,-x\}$ can be written in a unique way as $y=x\cos\theta+\omega\sin\theta$ for some $\omega\in E_x$ and $\theta\in(0,\pi)$. Moreover for any $f\in L^1(\S^{N-1})$
\begin{equation}\label{simplefor}
\int_{\S^{N-1}} f\,d\H^{N-1}=\int_{E_x}\int_0^\pi f(x\cos\theta+\omega\sin\theta)(\sin\theta)^{N-2}\,d\theta\,d\H^{N-2}(\omega)\,.
\end{equation}
We can then write, for every fixed $x\in\S^{N-1}$,
\[\begin{split}
\int_{\S^{N-1}} &g(y-x)\big(u(y)-u(x)\big)^2 \,d\H^{N-1}(y)\\
&=\int_{E_x}\int_0^\pi g(x\cos\theta+\omega\sin\theta-x)\Big(u(x\cos\theta+\omega\sin\theta)-u(x)\Big)^2(\sin\theta)^{N-2}\,d\theta\,d\H^{N-2}(\omega)\,.
\end{split}\]
Now, for any $\omega\in E_x$ and any $0<\theta<\pi$, recalling that $\omega\cdot x=0$ we evaluate
\[\begin{split}
\big(u(x\cos\theta+\omega&\sin\theta)-u(x)\big)^2
=\bigg(\int_0^1\theta\langle\nabla_\tau u(x\cos(s\theta)+\omega\sin(s\theta)),-x\sin(s\theta)+\omega\cos(s\theta)\rangle\,ds\bigg)^2\\
&\leq \theta^2 \int_0^1|\nabla_\tau u(x\cos(s\theta)+\omega\sin(s\theta))|^2ds
=\theta \int_0^\theta|\nabla_\tau u(x\cos\varphi+\omega\sin\varphi)|^2d\varphi\\
&\leq \theta \int_0^\pi|\nabla_\tau u(x\cos\varphi+\omega\sin\varphi)|^2d\varphi\,,
\end{split}\]
and since $g$ is radial and decreasing we have
\[
g(x\cos\theta+\omega\sin\theta-x) = g(\sqrt{2-2\cos\theta})\leq g(\theta/2)\,.
\]
Summarizing, we have
\[\begin{split}
\int_{\S^{N-1}} &g(y-x)\big(u(y)-u(x)\big)^2\, d\H^{N-1}(y)\\
&\leq \int_{E_x}\int_0^\pi g(\theta/2)(\sin\theta)^{N-2} \theta \int_0^\pi |\nabla_\tau u(x\cos\varphi+\omega\sin\varphi)|^2 \,d\varphi\,d\theta\, d\H^{N-2}(\omega)\\
&= C(g) \int_{E_x} \int_0^\pi |\nabla_\tau u(x\cos\varphi+\omega\sin\varphi)|^2 \,d\varphi\, d\H^{N-2}(\omega)\,,
\end{split}\]
where $C(g)$ is a constant which only depends on $g$, whose existence is ensured by~(\ref{necessary}). Therefore, using again~(\ref{simplefor}) and the fact that if $y=x\cos\varphi+\omega\sin\varphi$ with $\omega\in E_x$, then
\[
\sin^2\varphi = 1 - \cos^2\varphi = 1 - |x\cdot y|^2\geq 1 - |x\cdot y|\,,
\]
we calculate
\[\begin{split}
\iint_{\S^{N-1}\times\S^{N-1}} \hspace{-10pt}&\hspace{10pt} g(y-x)\big(u(y)-u(x)\big)^2\, dy\,dx\\
&\leq C(g) \int_{\S^{N-1}} \int_{E_x}\int_0^\pi |\nabla_\tau u(x\cos\varphi+\omega\sin\varphi)|^2 \,d\varphi\, d\H^{N-2}(\omega)\,d\H^{N-1}(x)\\
&\leq C(g) \int_{\S^{N-1}} \int_{E_x}\int_0^\pi \frac{|\nabla_\tau u(x\cos\varphi+\omega\sin\varphi)|^2 (\sin\varphi)^{N-2}}{\big(1 - |x\cdot y|\big)^{\frac{N-2}2}} \,d\varphi\, d\H^{N-2}(\omega)\,d\H^{N-1}(x)\\
&= C(g) \int_{\S^{N-1}}  \int_{\S^{N-1}} \frac{|\nabla_\tau u(y)|^2}{\big(1-|x\cdot y|\big)^{\frac{N-2}2}}\,d\H^{N-1}(y)\,d\H^{N-1}(x)\\
&= C(g)\int_{\S^{N-1}}|\nabla_\tau u(y)|^2d\H^{N-1}(y)\int_{\S^{N-1}}\frac 1{\big[1-|x\cdot \rm e_1|]^{\frac{N-2}{2}}}\,d\H^{N-1}(x) \\
&=C(g)C(N)\int_{\S^{N-1}}|\nabla_\tau u(y)|^2d\H^{N-1}_y\,,
\end{split}\]
where ${\rm e}_1$ is an arbitrary vector in $\S^{N-1}$. The proof is then concluded.
\end{proof}

For a given function $u\in W^{1,2}(\S^{N-1})$ with $u>-1$ everywhere, denoting by $E$ the set given by~(\ref{defEu}), we define now
\begin{align}\label{defE+E-}
E^+= E\setminus B\,, && E^-= B\setminus E\,,
\end{align}
so that
\begin{align*}
E^+ = \Big\{ \rho x,\, x\in\S^{N-1},\, 1\leq \rho < 1+u^+(x)\Big\}\,, &&
E^- = \Big\{ \rho x,\, x\in\S^{N-1},\, 1-u^-(x) < \rho < 1\Big\}\,,
\end{align*}
calling as usual $u^+=u\vee 0$ and $u^-=-u\vee 0$. Thanks to the above result, we deduce the following estimate.

\begin{lemma}[$\En(E^+,E^-)$ is ``negligible'']\label{+-neg}
Let $u\in W^{1,2}(\S^{N-1})$, with $|u|<1/2$. Then
\[
\En(E^+,E^-) \leq C \|u\|_{W^{1,2}}^2\,,
\]
where $C$ is a constant, only depending on $N$ and on $g$.
\end{lemma}
\begin{proof}
For every $z\in E^+$ and every $w\in E^-$, we write $x=z/|z|$ and $y=w/|w|$. Notice that $|z-w|\geq |y-x|/2$, thus since $g$ is radial and decreasing we have
\[
g(z-w) \leq \tilde g(y-x)\,,
\]
where we write for brevity, for every $v\in\R^N$, $\tilde g(v)=g(v/2)$. Observe that of course
\begin{equation}\label{tildegg}
\int_0^1 \tilde g(t)t^{N-1}\, dt \leq 2^N \int_0^1 g(s) s^{N-1}\, ds < +\infty
\end{equation}
by~(\ref{necessary}). Calling $\pi:\R^N\setminus \{0\}\to \S^{N-1}$ the projection on the unit sphere, we can then evaluate
\[\begin{split}
\En(E^+,E^-) &= \iint_{E^+\times E^-} g(z-w)\,dz\,dw\\
&\leq \iint_{\pi(E^+)\times\pi(E^-)} \int_{\rho=1}^{1+u^+(x)}\int_{\sigma=1-u^-(y)}^1 \tilde g(y-x)\rho^{N-1}\sigma^{N-1}\, d\rho\,d\sigma\, dy\,dx\\
&\leq 2^{N-1} \iint_{\pi(E^+)\times\pi(E^-)} u^+(x) u^-(y)\tilde g(y-x)\,dy\,dx\,.
\end{split}\]
Notice that, for every $x\in \pi(E^+)$ and $y\in\pi(E^-)$, we have $u^+(x)>0$ and $u^-(y)>0$, hence
\[
u^+(x) u^-(y) \leq \big(u^+(x)+u^-(y)\big)^2 = (u(x)-u(y))^2\,.
\]
Thus the above estimate can be continued as
\[\begin{split}
\En(E^+,E^-) &\leq 2^{N-1} \iint_{\pi(E^+)\times\pi(E^-)} (u(y)-u(x))^2\tilde g(y-x)\,dy\,dx\\
&\leq 2^{N-1} \iint_{\S^{N-1}\times\S^{N-1}} (u(y)-u(x))^2\tilde g(y-x)\,dy\,dx\leq C\| u\|_{W^{1,2}(\S^{N-1})}^2\,,
\end{split}\]
where in the last inequality we have used Lemma~\ref{genest} with $\tilde g$ in place of $g$, which is possible by~(\ref{tildegg}). Notice that the constant $C$ depends on $N$ and on $\tilde g$, then in turn on $N$ and on $g$.
\end{proof}

Since we will need to calculate integrals of $g$ over translated balls, it is useful to set $\psi:\R^+\times\R^+\to \R^+$ and $J:(-1/2,1/2)\to\R$ as
\begin{align}\label{defpsiJ}
\psi(a,b) = \int_{B(a)} g(|y-x|)\,dy \qquad \hbox{with } |x|=b\,, && J(\sigma)=\psi(1+\sigma,1)-\psi(1,1)\,.
\end{align}
It is simple to observe that $\psi$ is locally Lipschitz continuous outside the diagonal, but this is not helpful since we will need to use $\psi(a,b)$ with $a\approx b\approx 1$. However, the following weaker property will play a crucial role in our construction.
\begin{lemma}\label{estiK1}
There exists a constant $C=C(N,g)$ such that, for every $3/4\leq \rho\leq 5/4$ and every $-1/4\leq \tau\leq 1/4$ one has
\begin{equation}\label{claimK1}
\begin{array}{c}
\big|\psi(\rho+\tau,\rho)-\psi(\rho,\rho) -J(\tau)\big| \leq C |\rho-1|\,, \\
\big|\psi(1,1+\tau)-\psi(1,1)+J(\tau)\big| \leq C|\tau|\,, \\
|J(\tau)+J(-\tau)| \leq C|\tau|\,.
\end{array}
\end{equation}
\end{lemma}
\begin{proof}
The thesis will follow from three main estimates. To start, we take $1/2\leq r,\,r'\leq 3/2$, and we show that $|r-r'|$ controls $|\psi(r,r)-\psi(r',r')|$. Without loss of generality we assume that $r> r'$. Notice that $\psi(r,r)-\psi(r',r')$, by definition, is the integral of $g$ on the set $A(r,r')$ given by the difference of two balls, a bigger one with radius $r$ and a smaller one with radius $r'$, being the smaller one contained in the bigger one and internally tangent.
\begin{figure}[htbp]
\begin{tikzpicture}
\fill[blue!25!white] (-0.76,-2) arc (-41.81:41.81:3) -- (-0.34,2) arc (19.47:-19.47:6);
\draw[line width=1] (-.76,-2) arc (-41.81:41.81:3);
\draw[line width=1] (-.34,-2) arc (-19.47:19.47:6);
\fill (0,0) circle (1.8pt);
\draw (0,0) node[anchor=north west]{$O$};
\draw[angle 90-angle 90] (-4,0) -- (-0.08,0);
\draw (-2,0) node[anchor=south]{$r'$};
\draw (-3,-0.2) node[anchor=north]{$r$};
\draw[angle 90-angle 90] (-6,-0.2) -- (-0.08,-0.2);
\draw[angle 90-angle 90] (-0.75,0.2) -- (-.08,.2);
\draw (-0.375,0.2) node[anchor=south]{$\sigma$};
\draw[line width=1] (1.5,0) arc (0:105:1.5);
\draw[angle 90-angle 90] (0.08,0) -- (1.42,0);
\draw (0.75,0) node[anchor=north]{$t$};
\fill (-0.75,0) circle (1.8pt);
\draw (-0.75,0) node[anchor=south east]{$S$};
\fill (1.5,0) circle (1.8pt);
\draw (1.5,0) node[anchor=north west]{$R$};
\fill (-0.18,1.49) circle (1.8pt);
\draw (-0.18,1.49) node[anchor=south west]{$P$};
\fill (-0.39,1.45) circle (1.8pt);
\draw (-0.39,1.45) node[anchor=north east]{$Q$};
\draw[line width=0.5] (.4,0) arc (0:92:.4);
\draw (0.2,0.2) node[anchor=south west]{$\theta$};
\end{tikzpicture}
\caption{The (coloured) set $A(r,r')$ and the angle $\theta$ in the proof of~(\ref{biges1}) and~(\ref{biges2}).}\label{FigK1}
\end{figure}
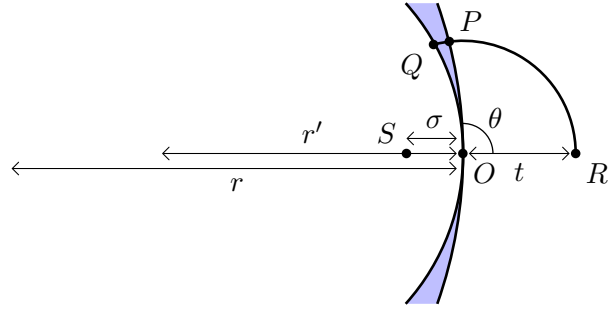
Figure~\ref{FigK1} shows the set $A(r,r')$, coloured, close to the point of tangency, that we consider to be the origin $O$. We also consider the exterior normal to the two balls in the tangency point to be horizontal (i.e., parallel to the first vector of a given orthonormal basis). Let us assume for a moment that $N=2$, just for simplicity in the figure. As shown in the figure, we fix $0<t<1/4$, and we call $R$ the point having distance $t$ from $O$ in the horizontal direction. We consider then the circle $S_2(t)$ with radius $t$ centered at $O$, we call $P$ one of the two points of intersection of $S_2(t)$ with the larger ball, and we denote by $\theta$ the angle $\angle ROP$. In the very same way, we call $Q$ a point of intersection between $S_2(t)$ and the smaller ball, and we call $\theta'$ the angle $\angle ROQ$. One readily has that $\cos\theta= -t/2r$, and similarly $\cos\theta'=-t/2r'$. Since by geometric reasons $\frac\pi 2<\theta<\theta'<\frac 34\,\pi$ because we are considering $0<t<1/4$, we get
\[
\theta'-\theta\leq 2 (\cos\theta-\cos\theta') = \frac{t(r-r')}{rr'}\leq 4 t(r-r')\,.
\]
Therefore, in the $2$-dimensional case, we can estimate for every $0<t<1/4$
\[
\H^1\big(A(r,r')\cap S_2(t)\big) = 2t(\theta'-\theta) \leq 8t^2(r-r')\,.
\]
Let us pass to the general $N$-dimensional case. Calling $A_2(r,r')$ the $2$-dimensional set already studied, we have in general
\[
A(r,r')= \Big\{ (z_1,z')\in \R\times\R^{N-1}:\, (z_1,|z'|)\in A_2(r,r') \Big\}\,.
\]
Calling then $S_N(t)$ the sphere with radius $t$ centered at $0$, an immediate integration in cylindrical coordinates gives, for every $0<t<1/4$,
\[
\H^{N-1}\big(A(r,r')\cap S_N(t)\big) \leq (N-1)\omega_{N-1} t^{N-2} \H^1\big(A_2(r,r')\cap S_2(t)\big) \leq 8 (N-1)\omega_{N-1}  t^N(r-r')\,.
\]
We have then
\begin{equation}\label{firstcal}\begin{split}
\psi(r,r)&-\psi(r',r') = \int_{A(r,r')} g(w)\,dw = \int_{t=0}^3 g(t) \H^{N-1}\big(A(r,r')\cap S_N(t)\big)\,dt\\
&\leq \int_0^{1/4} g(t) \H^{N-1}\big(A(r,r')\cap S_N(t)\big)\,dt+g(1/4)\int_{1/4}^3\H^{N-1}\big(A(r,r')\cap S_N(t)\big)\,dt\\
&\leq 8(N-1)\omega_{N-1}(r-r') \int_0^{1/4} g(t) t^N \,dt+ g(1/4) \big|A(r,r')\big|\leq C(r-r')\,,
\end{split}\end{equation}
where $C$ is a constant only depending on $N$ and on $g$. For every $1/2\leq r,\,r'\leq 3/2$ we have then
\begin{equation}\label{biges1}
\big|\psi(r,r)-\psi(r',r')\big| \leq C |r-r'|\,.
\end{equation}
We pass now to the second main estimate. Let us take $-1/4\leq \sigma\leq 1/4$ and let us show that $|r-r'|$ controls also $|\psi(r,r-\sigma)-\psi(r',r'-\sigma)|$. As before, without loss of generality we can assume that $r>r'$. The value of the difference $\psi(r,r-\sigma)-\psi(r',r'-\sigma)$ is then exactly as before given by an integral over the set $A(r,r')$. The only difference is that this time the function to integrate is not $g(w)$, but $g(w-S)$, where $S$ is the point having distance $\sigma$ from $O$ in the horizontal, negative direction. Figure~\ref{FigK1} shows the point $S$ in the case when $\sigma>0$. Notice that the points of $A(r,r')$ close to $O$ are much closer to $O$ than to $S$. More in general, a trivial geometric argument ensures that for every $w\in A(r,r')$ one has
\[
|w|=|w-O| \leq 2 |w-S|\,,
\]
the constant $2$ is actually not needed if $\sigma<0$. As a consequence, we have
\[
\psi(r,r-\sigma)-\psi(r',r'-\sigma) = \int_{A(r,r')} g(w-S)\,dw \leq \int_{A(r,r')} \tilde g(w)\,dw\,,
\]
where as in the proof of Lemma~\ref{+-neg} we write $\tilde g(w)=g(w/2)$. The same calculation as in~(\ref{firstcal}), keeping in mind~(\ref{tildegg}), gives then that for every $1/2\leq r,\,r'\leq 3/2$ and every $-1/4\leq \sigma\leq 1/4$
\begin{equation}\label{biges2}
\big|\psi(r,r-\sigma)-\psi(r',r'-\sigma)\big| \leq C |r-r'|\,.
\end{equation}
Let us finally pass to the third and last main estimate, which consists in taking again $-1/4\leq \sigma\leq 1/4$, and showing that $|\sigma|$ controls $|J(\sigma)+J(-\sigma)|$. Without loss of generality let us assume that $\sigma>0$. 
\begin{figure}[htbp]
\begin{tikzpicture}
\fill[blue!15!white] (0.2,-3) arc (-30:30:6) -- (-1,3) arc (36.9:-36.9:5);
\fill[blue!25!white] (-1,3) arc (36.9:-36.9:5) -- (-2.35,-3) arc (-48.6:48.6:4);
\draw[line width=1.2] (1,0) arc (0:30:6);
\draw[line width=1.2] (1,0) arc (0:-30:6);
\draw[line width=1.2] (0,0) arc (0:36.9:5);
\draw[line width=1.2] (0,0) arc (0:-36.9:5);
\draw[line width=1.2] (-1,0) arc (0:48.6:4);
\draw[line width=1.2] (-1,0) arc (0:-48.6:4);
\fill (0,0) circle (1.8pt);
\draw (0,0) node[anchor=north east]{$O$};
\draw[angle 90-angle 90] (-.97,0) -- (-.06,0);
\draw[angle 90-angle 90] (.97,0) -- (.06,0);
\draw (-.5,0) node[anchor=south]{$\sigma$};
\draw (.5,0) node[anchor=south]{$\sigma$};
\draw[angle 90-angle 90] (0.06,-.4) -- (2.5,-.4);
\draw (1.25,-.4) node[anchor=north]{$t$};
\draw (-.3,-2.75) node{$A(\sigma)$};
\draw (-1.5,-2.75) node{$A(-\sigma)$};
\draw[line width=1] (2.5,0) arc (0:127.5:2.5);
\fill (-1.52,1.98) circle (1.8pt);
\fill (-.63,2.42) circle (1.8pt);
\fill (.46,2.46) circle (1.8pt);
\draw (-1.52,1.98) node[anchor=north east]{$P^-$};
\draw (-.63,2.42) node[anchor=south west]{$P$};
\draw (.46,2.46) node[anchor=south west]{$P^+$};
\fill[blue!15!white] (-6.38,-3) arc (-14.48:14.48:12) -- (-8.46,3) arc (17.46:-17.46:10);
\fill[blue!25!white] (-8.46,3) arc (17.46:-17.46:10) -- (-10.58,-3) arc(-22.02:22.02:8);
\draw[line width=1.2] (-8,0) arc (0:17.46:10);
\draw[line width=1.2] (-8,0) arc (0:-17.46:10);
\draw[line width=1.2] (-6,0) arc (0:14.48:12);
\draw[line width=1.2] (-6,0) arc (0:-14.48:12);
\draw[line width=1.2] (-10,0) arc (0:22.02:8);
\draw[line width=1.2] (-10,0) arc (0:-22.02:8);
\draw[angle 90-angle 90] (-7.94,0) -- (-6.03,0);
\draw[angle 90-angle 90] (-9.97,0) -- (-8.06,0);
\fill (-8,0) circle (1.8pt);
\draw (-8,0) node[anchor=north east]{$O$};
\draw (-7,0) node[anchor=south]{$\sigma$};
\draw (-9,0) node[anchor=south]{$\sigma$};
\draw[dashed] (-8,-3) -- (-8,3);
\draw[dashed] (-10,-3) -- (-10,3);
\draw[dashed] (-6,-3) -- (-6,3);
\draw[angle 90-angle 90] (-5.8,0) -- (-5.8,2.5);
\draw (-5.8,1.25) node[anchor=west]{$t$};
\draw[dotted] (-10.4,2.5) -- (-6,2.5);
\fill (-6.25,2.5) circle (1.8pt);
\draw[angle 90-angle 90] (-6.25,2.7) -- (-6,2.7);
\fill (-8.32,2.5) circle (1.8pt);
\draw[angle 90-angle 90] (-8.32,2.7) -- (-8,2.7);
\fill (-10.4,2.5) circle (1.8pt);
\draw[angle 90-angle 90] (-10.4,2.7) -- (-10,2.7);
\draw[-angle 90] (-6.8,1.7) -- (-6.125,2.4);
\draw[-angle 90] (-6.8,1.7) -- (-8.16,2.4);
\draw[-angle 90] (-6.8,1.7) -- (-10.2,2.4);
\draw (-6.8,1.7) -- (-6.8,1.5);
\draw (-7.7,1) node[anchor=south west]{$\approx t^2/2$};
\draw (-7.3,-2.75) node{$A(\sigma)$};
\draw (-9.3,-2.75) node{$A(-\sigma)$};
\end{tikzpicture}
\caption{The (coloured) sets $A(\sigma)$ and $A(-\sigma)$ and the situation in the proof of~(\ref{biges3}).}\label{FigK12}
\end{figure}
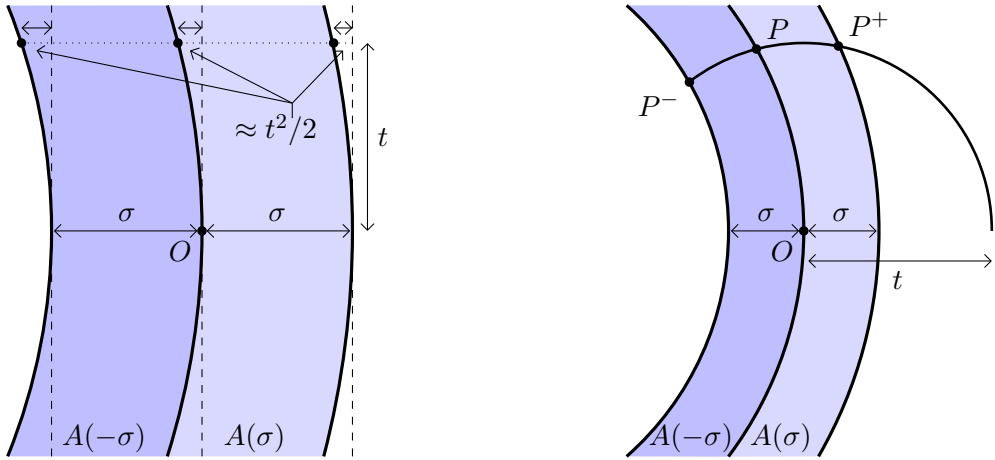
Observe that $J(\sigma)=\psi(1+\sigma,1)-\psi(1,1)$ is the integral of $g$ over an annulus $A(\sigma)$ with radii $1$ and $1+\sigma$, the origin being at the internal boundary, while $-J(-\sigma)=\psi(1,1)-\psi(1-\sigma,1)$ is the integral of $g$ over an annulus $A(-\sigma)$ with radii $1$ and $1-\sigma$, the origin being at the external boundary. Figure~\ref{FigK12} shows the annuli $A(\sigma)$ and $A(-\sigma)$ near $O$ with two different magnifications. Let us start near the origin $O$, noticing that $A(\sigma)$ and $A(-\sigma)$ are close to the slabs
\begin{align*}
C^+=\big\{(z_1,z')\in\R\times\R^{N-1}:\, 0<z_1<\sigma\big\}\,, && C^-=\big\{(z_1,z')\in\R\times\R^{N-1}:\, -\sigma<z_1<0\big\}\,.
\end{align*}
More precisely, fix any $0<t<2\sigma$, and set $S(t) = B(2\sigma) \cap\{(z_1,z'):\, |z'|=t\}$.
Since of course
\[
\int_{C^+\cap S(t)} g(w)\,d\H^{N-1}(w)=\int_{C^-\cap S(t)} g(w)\,d\H^{N-1}(w)\,,
\]
keeping in mind that $g$ is decreasing and by an immediate geometric argument (see Figure~\ref{FigK12} left) we can estimate
\[\begin{split}
\bigg|\int_{A(\sigma) \cap S(t)}& g(w)\,d\H^{N-1}(w) - \int_{A(-\sigma) \cap S(t)} g(w)\,d\H^{N-1}(w)\bigg|\\
&\leq g(t) \Big( \H^{N-1}\big(\big(A(\sigma)\Delta C^+\big) \cap S(t)\big) + \H^{N-1}\big(\big(A(-\sigma)\Delta  C^-\big)\cap S(t)\big)\Big)\\
&\leq 4(N-1)\omega_{N-1}t^N g(t)\,.
\end{split}\]
By integrating in $t$, then, we have
\begin{equation}\label{estinB2s}\begin{split}
\bigg|\int_{A(\sigma)\cap B(2\sigma)} g(w)\, dw &-\int_{A(-\sigma)\cap B(2\sigma)} g(w)\, dw\bigg|
\leq \int_{t=0}^{2\sigma} 4(N-1)\omega_{N-1} t^N g(t)\,dt \\
&\leq 8 (N-1)\omega_{N-1} \sigma \int_0^{2\sigma} g(t)t^{N-1}\,dt
 \leq C\sigma\,.
\end{split}\end{equation}
Let us now pass to consider the situation outside the ball $B(2\sigma)$. As in the proof of~(\ref{biges1}), we call $S_N(t)$ the sphere with radius $t$ centered at $0$, and we start considering the situation in the $2$-dimensional case, with circle $S_2(t)$ and annuli $A_2(\pm \sigma)$. Let us fix any $2\sigma<t<1/4$. As depicted in Figure~\ref{FigK12}, right, the circle $S_2(t)$ intersects $A_2(\sigma)$ in two symmetric arcs, and the same is true for the intersection with $A_2(-\sigma)$. Let us call $P,\,P^+$ and $P^-$ three intersection points, as in the figure, and let us call $\theta,\,\theta^+$ and $\theta^-$ the directions of the segments $OP,\, OP^+$ and $OP^-$. Notice that $\theta^+<\theta<\theta^-$, and the three directions are close to $\pi/2$ when $\sigma\ll t\ll 1$. A very simple trigonometric calculation ensures that
\begin{align}\label{vstc}
\cos\theta=-\frac t2\,, && \cos\theta^+ = \frac{-t^2+2\sigma+\sigma^2}{2t}\,, && \cos\theta^- = \frac{-t^2-2\sigma+\sigma^2}{2t}\,,
\end{align}
and since $2\sigma<t<1/4$ this implies
\begin{align*}
\theta<\theta^- <\frac 34\,\pi \,, &&\frac \pi 2<\theta< \frac 7{12} \pi\,, && \frac\pi 3<\theta^+<\theta\,,
\end{align*}
in particular $\theta^--\theta$ and $\theta-\theta^+$ are both smaller than $\pi/4$, so that
\begin{equation}\label{bddlng}\begin{split}
\big| \theta^++\theta^--2\theta\big| &=\big|(\theta^--\theta) - (\theta-\theta^+)\big|
\leq \sqrt 2 \big|\sin(\theta^--\theta) -\sin (\theta-\theta^+) \big|\\
&\leq 2 \sin\theta \big|\sin(\theta^--\theta) -\sin (\theta-\theta^+) \big|\\
&= 2 \Big| \cos\theta^++\cos\theta^--2\cos\theta-\cos\theta\Big(\cos(\theta^--\theta)+\cos(\theta-\theta^+)-2\Big)\Big|\\
&\leq 2\,\frac {\sigma^2}t+ t\Big|\big(\cos(\theta^--\theta)+\cos(\theta-\theta^+)-2\big)\Big|\\
&\leq 2\,\frac {\sigma^2}t+ \frac t 2 \Big((\theta^--\theta)^2+(\theta-\theta^+)^2\Big)\\
&\leq 2\,\frac {\sigma^2}t+  t \Big((\cos\theta^--\cos\theta)^2+(\cos\theta-\cos\theta^+)^2\Big)
\leq 5\,\frac {\sigma^2}t\leq 3\sigma\,.
\end{split}\end{equation}
We can now calculate
\[\begin{split}
\H^{N-1}\big(A(\sigma)&\cap S_N(t)\big)
=  \int_{\alpha=\theta^+}^\theta (N-1)\omega_{N-1} (t\sin\alpha)^{N-2} t\,d\alpha\\
&= (N-1)\omega_{N-1} t^{N-1} \int_{\alpha=\theta^+}^\theta (\sin\alpha)^{N-2}\,d\alpha\\
&= (N-1)\omega_{N-1} t^{N-1} \bigg((\sin\theta)^{N-2}(\theta-\theta^+) + \int_{\alpha=\theta^+}^\theta (\sin\alpha)^{N-2}-(\sin\theta)^{N-2}\,d\alpha\bigg)\,,
\end{split}\]
and similarly
\[
\frac{\H^{N-1}\big(A(-\sigma)\cap S_N(t)\big)}{(N-1)\omega_{N-1}}=  t^{N-1} \bigg((\sin\theta)^{N-2}(\theta^--\theta) + \int_{\alpha=\theta}^{\theta^-} (\sin\alpha)^{N-2}-(\sin\theta)^{N-2}\,d\alpha\bigg)\,,
\]
so that
\begin{equation}\label{partcal}
\frac{\Big|\H^{N-1}\big(A(\sigma)\cap S_N(t)\big)-\H^{N-1}\big(A(-\sigma)\cap S_N(t)\big)\Big|}{(N-1)\omega_{N-1}} \leq
 t^{N-1} \Big(\big|\theta^++\theta^--2\theta\big| + K\Big)\,,
\end{equation}
where
\[
K=\bigg|\int_{\theta^+}^\theta (\sin\alpha)^{N-2}-(\sin\theta)^{N-2}\,d\alpha - \int_\theta^{\theta^-} (\sin\alpha)^{N-2}-(\sin\theta)^{N-2}\,d\alpha\bigg|\,.
\]
We claim that
\begin{equation}\label{boundK}
K\leq 9(N-2) \sigma\,.
\end{equation}
To show this inequality, we first observe that by~(\ref{vstc}) we have
\begin{align}\label{stret}
|\theta^+-\theta|\leq \sqrt 2 |\cos\theta^+-\cos\theta|\leq  2\, \frac\sigma t\,, && |\theta^--\theta|\leq \sqrt 2 |\cos\theta^--\cos\theta|\leq 2\, \frac\sigma t\,.
\end{align}
We distinguish then two cases. If $t\geq\sqrt \sigma$, then again by~(\ref{vstc}) we have $|\cos\theta^+|,\,|\cos\theta^-|\leq 2t$, thus for every $\theta^+<\alpha<\theta^-$ by~(\ref{stret}) one has
\[
|(\sin\alpha)^{N-2}-(\sin\theta)^{N-2}| \leq (N-2)|\sin\alpha-\sin\theta| \leq 2(N-2)t |\alpha-\theta|\leq 4(N-2) \sigma \,,
\]
so that
\[
K \leq 16(N-2)\,\frac{\sigma^2 }t\leq 8(N-2) \sigma\,,
\]
and then~(\ref{boundK}) is proved in the case $t\geq \sqrt\sigma$. Suppose insted that $2\sigma<t<\sqrt\sigma$. In this case, for every $\theta^+<\alpha<\theta^-$ by~(\ref{stret}) we have that
\begin{equation}\label{inthca}
|(\sin\alpha)^{N-2}-(\sin\theta)^{N-2}| \leq (N-2)|\sin\alpha-\sin\theta| \leq (N-2) |\alpha-\theta|\leq 2(N-2)\,\frac\sigma t \,.
\end{equation}
Let us now call $\hat\theta$ and $\hat\theta^+$ the directions obtained by a vertical mirror symmetry of $\theta$ and $\theta^+$, that is, $\hat\theta=\pi-\theta$ and $\hat\theta^+=\pi-\theta^+$. Observe that, again by~(\ref{vstc}) and since $t<\sqrt\sigma$, we have $\theta^+<\hat\theta<\pi/2 < \theta<\hat\theta^+<\theta^-$. Since by symmetry we have
\[
\int_{\theta^+}^{\hat\theta} (\sin\alpha)^{N-2}-(\sin\theta)^{N-2}\,d\alpha = \int^{\hat\theta^+}_\theta (\sin\alpha)^{N-2}-(\sin\theta)^{N-2}\,d\alpha\,,
\]
by~(\ref{inthca}) and~(\ref{bddlng}) we have
\[\begin{split}
K &\leq 2(N-2)\,\frac\sigma t \Big( (\theta-\hat\theta) + (\theta^--\hat\theta^+)\Big)= 2(N-2)\,\frac\sigma t \Big( 2 (\theta-\hat\theta) + (\theta^++\theta^--2\theta)\Big)\\
&\leq 2(N-2)\,\frac\sigma t \Big( 2\sqrt 2 (\cos\hat\theta-\cos\theta) + 3\sigma\Big)
= 2(N-2)\,\frac\sigma t \Big( 2\sqrt 2 t + 3\sigma\Big)
\leq 9(N-2) \sigma\,,
\end{split}\]
thus~(\ref{boundK}) is proved also in the case $t<\sqrt\sigma$. Inserting~(\ref{boundK}) into~(\ref{partcal}) and keeping in mind~(\ref{bddlng}), we have then for every $2\sigma<t<1/4$
\[
\Big|\H^{N-1}\big(A(\sigma)\cap S_N(t)\big)-\H^{N-1}\big(A(-\sigma)\cap S_N(t)\big)\Big| \leq C t^{N-1} \sigma\,.
\]
Putting together this inequality and~(\ref{estinB2s}), we obtain the third main estimate, that is,
\begin{equation}\label{biges3}\begin{split}
\big|J(\sigma)&+J(-\sigma)\big|= \bigg|\int_{A(\sigma)} g(w)\,dw - \int_{A(-\sigma)} g(w)\,dw\bigg| \\
&\leq C\sigma + \int_{t=2\sigma}^3 g(t) \Big| \H^{N-1}\big(A(\sigma)\cap S_N(t)\big)-\H^{N-1}\big(A(-\sigma)\cap S_N(t)\big)\Big|\,dt\\
&\leq C\sigma +C \int_{t=2\sigma}^{1/4} g(t) t^{N-1} \sigma\,dt+ g(1/4)\Big( \big|A(\sigma)\big|+\big|A(-\sigma)\big|\Big)
\leq C\sigma\,.
\end{split}\end{equation}
Thanks to the main estimates~(\ref{biges1}), (\ref{biges2}) and~(\ref{biges3}), it is immediate to prove~(\ref{claimK1}). The third estimate in~(\ref{claimK1}) is simply~(\ref{biges3}) with $\sigma=|\tau|$. The first estimate in~(\ref{claimK1}) comes by putting together~(\ref{biges2}) with $r=\rho+\tau$, $r'=1+\tau$ and $\sigma=\tau$, and~(\ref{biges1}) with $r=\rho$ and $r'=1$, getting
\[\begin{split}
\big|\psi(\rho+\tau,\rho)-\psi(\rho,\rho) -J(\tau)\big|&=
\big|\psi(\rho+\tau,\rho)-\psi(\rho,\rho) -\psi(1+\tau,1)+\psi(1,1)\big| \\
&\leq \big|\psi(\rho+\tau,\rho) -\psi(1+\tau,1)\big| +|\psi(\rho,\rho)-\psi(1,1)|\leq C|\rho-1|\,.
\end{split}\]
Finally, the second estimate in~(\ref{claimK1}) comes by putting together~(\ref{biges2}) with $r=1$, $r'=1-\tau$ and $\sigma=-\tau$, and~(\ref{biges3}) with $\sigma=|\tau|$, obtaining
\[
\big|\psi(1,1+\tau)-\psi(1,1)+J(\tau)\big| \leq \big|\psi(1,1+\tau)-\psi(1-\tau,1)\big|+\big|J(-\tau)+J(\tau)\big|\leq C|\tau|\,.
\]
The proof is then concluded.
\end{proof}

We are now ready to give the proof of our main result.

\proofof{Theorem~\mref{main}}
Let $\eps>0$ be given, and let $E$ be a minimizer of $\F_\eps$ among sets of volume $\omega_N$. We already know by Lemma~\ref{lemreg} that, if $\eps$ is small enough, then up to a translation $E$ is of the form $E(u)$ given by~(\ref{defEu}) for a uniformly small function $u\in W^{1,2}(\S^{N-1})$. Let $E^+$ and $E^-$ be defined as in~(\ref{defE+E-}), and notice that the sets $E^+\subseteq \R^N\setminus B$ and $E^-\subseteq B$ have the same volume, and are done by points uniformly close to the sphere $\S^{N-1}$. We can write
\begin{equation}\label{calcener}
\En(E)-\En(B) = 2\En(B,E^+) - 2\En(B,E^-) + \En(E^+)+\En(E^-)-2\En(E^+,E^-)\,.
\end{equation}
Using then the notation introduced in~(\ref{defpsiJ}), we can also calculate
\begin{equation}\label{1stpce}\begin{split}
\En(B,E^+)-\En(B,E^-)&=
\iint_{B\times E^+} g(z-w)\,dz\,dw - \iint_{B\times E^-} g(z-w)\,dz\,dw\\
&=\int_{E^+} \psi(1,|z|)\, dz - \int_{E^-} \psi(1,|z|)\,dz\\
&=\int_{E^+} \psi(1,|z|)-\psi(1,1)\, dz - \int_{E^-} \psi(1,|z|)-\psi(1,1)\,dz\,.
\end{split}\end{equation}
Let us now observe that, also by Lemma~\ref{estiK1},
\[\begin{split}
\int_{E^+} \psi(1,|z|)-\psi(1,1)\, dz
&=\int_{x\in\partial B} \int_{t=0}^{u^+(x)} (1+t)^{N-1} \Big(\psi(1,1+t)-\psi(1,1)\Big)\,dt\,dx\\
&=\int_{\partial B} \int_0^{u^+(x)} \psi(1,1+t)-\psi(1,1)\,dt\,dx + O(\|u\|_{L^2}^2)\\
&= -\int_{\partial B} \int_0^{u^+(x)} J(t) \,dt\,dx + O(\|u\|_{L^2}^2)\,,
\end{split}\]
and in the very same way
\[
\int_{E^-} \psi(1,|z|)-\psi(1,1)\, dz= -\int_{\partial B} \int_0^{u^-(x)} J(-t) \,dt\,dx + O(\|u\|_{L^2}^2)\,.
\]
The equality~(\ref{1stpce}) becomes then
\[
\En(B,E^+)-\En(B,E^-) =  -\int_{\partial B} \int_0^{u^+(x)} J(t) \,dt\,dx +\int_{\partial B} \int_0^{u^-(x)} J(-t) \,dt\,dx + O(\|u\|_{L^2}^2)\,,
\]
which inserted in~(\ref{calcener}) and recalling Lemma~\ref{+-neg} gives
\begin{equation}\label{calcener2}\begin{split}
\En(E)-\En(B)&= \En(E^+) -2\int_{\partial B} \int_0^{u^+(x)} J(t) \,dt\,dx\\ &\qquad+\En(E^-) +2\int_{\partial B} \int_0^{u^-(x)} J(-t) \,dt\,dx + O(\|u\|_{W^{1,2}}^2)\,.
\end{split}\end{equation}
In order to evaluate $\En(E^+)$ and $\En(E^-)$, we call for brevity
\[
\varphi(x,y,s,t)=(1+t)^{N-1} (1+s)^{N-1} g\big((1+t)x-(1+s)y\big)\,,
\]
so that by definition
\[\begin{split}
\En(E^+)&=\int_{x\in\partial B} \int_{t=0}^{u^+(x)} \int_{y\in\partial B} \int_{s=0}^{u^+(y)} \varphi(x,y,s,t)\,ds\,dy\,dt\,dx\\
&=\int_{\partial B} \int_0^{u^+(x)} \int_{\partial B} \int_0^{u^+(x)} \varphi(x,y,s,t)\,ds\,dy\,dt\,dx\\
&\qquad\qquad +\int_{\partial B} \int_0^{u^+(x)} \int_{\partial B} \int_{u^+(x)}^{u^+(y)}  \varphi(x,y,s,t)\,ds\,dy\,dt\,dx= K_1 + K_2\,,
\end{split}\]
where $K_1$ and $K_2$ denote the two terms of the last equality.\par

Let us start working on $K_2$. As in the proof of Lemma~\ref{+-neg}, we can define $\tilde g(v)=g(v/2)$ for every $v\in\R^N$, and observe that for every $x,\,y\in \partial B$ and $s,\,t\in (-1/2,1/2)$ one has
\[
g\big((1+t)x-(1+s)y\big) \leq \tilde g(y-x)\,.
\]
As a consequence, for every pair $x,\,y\in\partial B$, we can estimate
\[\begin{split}
\int_0^{u^+(x)} \int_{u^+(x)}^{u^+(y)}  &\varphi(x,y,s,t)\,ds\,dt+\int_0^{u^+(y)} \int_{u^+(y)}^{u^+(x)}  \varphi(x,y,s,t)\,ds\,dt\\
&= - \int_{u^+(x)}^{u^+(y)}  \int_{u^+(x)}^{u^+(y)} \varphi(x,y,s,t)\,ds\,dt\geq -\bigg(\frac 32\bigg)^{2N-2} \int_{u^+(x)}^{u^+(y)}  \int_{u^+(x)}^{u^+(y)}\tilde g(y-x)\,ds\,dt\,.
\end{split}\]
Inserting this estimate in the definition of $K_2$, and applying again Lemma~\ref{genest} with $\tilde g$ in place of $g$, which is admissible by~(\ref{tildegg}), we have
\begin{equation}\label{claimK2}
K_2 \geq -\frac{3^{2N-2}}{2^{2N-1}}\, \int_{\partial B}\int_{\partial B} \big(u^+(y)-u^+(x)\big)^2\tilde g(y-x)\geq  - C \|u\|_{W^{1,2}}^2\,,
\end{equation}
where as usual $C$ is a constant depending only on $N$ and $g$.\par
Let us now pass to evaluate $K_1$, which can be rewritten as
\[\begin{split}
K_1&=\int_{\partial B} \int_0^{u^+(x)} (1+t)^{N-1} \int_{B(1+u^+(x))\setminus B(1)} g\big((1+t)x-w\big)\,dw\,dt\,dx\\
&=\int_{\partial B} \int_0^{u^+(x)} (1+t)^{N-1}  \Big(\psi(1+u^+(x),1+t)-\psi(1,1+t)\Big)\,dt\,dx\\
&=\int_{\partial B} \int_0^{u^+(x)} \psi(1+u^+(x),1+t)-\psi(1,1+t)\,dt\,dx+ O(\|u\|_{L^2}^2)\,.
\end{split}\]
Rewriting $\psi(1+u^+(x),1+t)-\psi(1,1+t)$ as
\[
\psi(1+u^+(x),1+t)-\psi(1+t,1+t)+\psi(1+t,1+t)-\psi(1,1+t)
\]
and keeping in mind Lemma~\ref{estiK1}, we obtain
\[\begin{split}
K_1&\geq \int_{\partial B} \int_0^{u^+(x)} J(u^+(x)-t) +J(t) -3C u^+(x)\,dt\,dx+ O(\|u\|_{L^2}^2)\\
&=\int_{\partial B} \int_0^{u^+(x)} J(u^+(x)-t) +J(t) \,dt\,dx+ O(\|u\|_{L^2}^2)
=2\int_{\partial B} \int_0^{u^+(x)}J(t) \,dt\,dx+ O(\|u\|_{L^2}^2)\,.
\end{split}\]
Since $\En(E^+)=K_1+K_2$, this equality and~(\ref{claimK2}) give
\[
\En(E^+) \geq 2\int_{\partial B} \int_0^{u^+(x)}J(t) \,dt\,dx- C \|u\|_{W^{1,2}}^2\,.
\]
The very same calculations with $E^-$ in place of $E^+$ give
\[
\En(E^-) \geq -2\int_{\partial B} \int_0^{u^-(x)}J(-t) \,dt\,dx- C \|u\|_{W^{1,2}}^2\,.
\]
Putting these last two estimates into~(\ref{calcener2}), we have then $\En(E)-\En(B)\geq - C \|u\|_{W^{1,2}}^2$. By~(\ref{marcogippo}), we have then
\[
\F_\eps(E) \geq \F_\eps(B) + \big(C_N - \eps C\big) \|u\|_{W^{1,2}}^2\,,
\]
hence of course the unique minimizer of the energy $\F_\eps$ is the ball $B$ as soon as $\eps\ll 1$.
\end{proof}


\begin{thebibliography}{99}

\bibitem{BC} M. Bonacini \& R. Cristoferi, Local and global minimality results for a nonlocal isoperimetric problem on $\R^N$, SIAM J. Math. Anal. {\bf 46} (2014), no. 4, 2310--€"2349.
\bibitem{CL} M. Cicalese \& G. P. Leonardi. A selection principle for the sharp quantitative isoperimetric
inequality. Arch. Ration. Mech. Anal. {\bf 206} (2012), no. 2, 617--643.
\bibitem{F2M3} A. Figalli, N. Fusco, F. Maggi, V. Millot \& M. Morini, Isoperimetry and stability properties of balls with respect to nonlocal energies, Comm. Math. Phys. {\bf 336} (2015), no. 1, 441--507.
\bibitem{F} B. Fuglede, Stability in the isoperimetric problem for convex or nearly spherical domains in $\R^{n}$. Trans. Amer. Math. Soc. {\bf 314} (1989), no. 2, 619--638.
\bibitem{J} V. Julin, Isoperimetric problem with a Coulomb repulsive term, Indiana Univ. Math. J. {\bf 63} (214), no. 1, 77--89.
\bibitem{KM1} H. Kn\"upfer \& C. Muratov,  On an isoperimetric problem with a competing nonlocal term I: The planar case, Comm. Pure Appl. Math. {\bf 66} (2013), no. 7, 1129--1162. 
\bibitem{KM2} H. Kn\"upfer \& C. Muratov, On an isoperimetric problem with a competing nonlocal term II: The general case, Comm. Pure Appl. Math. {\bf 67} (2014), no. 12, 1974--1994. 
\bibitem{M} F. Maggi, Sets of finite perimeter and geometric variational problems, volume 135 of
Cambridge Studies in Advanced Mathematics. Cambridge University Press, Cambridge, 2012.
An introduction to geometric measure theory.
\bibitem{NP} M. Novaga \& A. Pratelli, Minimisers of a general Riesz-type Problem, preprint (2020).


\end{thebibliography}
\end{document}